\documentclass[11pt]{article}
\usepackage{amssymb}
\usepackage{mathrsfs}
\usepackage{latexsym,amsmath,amssymb,amsfonts,epsfig,graphicx,cite,psfrag}
\usepackage{eepic,color,colordvi,amscd}
\usepackage{ebezier}
\usepackage{verbatim}
\usepackage{amsthm}
\usepackage{amssymb}
\usepackage{epsf,epsfig,amsfonts,amsgen,amsmath,amstext,amsbsy,amsopn,amsthm
}
\usepackage{amssymb}
\usepackage{ebezier,eepic}
\usepackage{color}
\usepackage{multirow}
\usepackage{comment}

\newtheorem{theo}{Theorem}[section]

\newtheorem{ques}[theo]{Question}
\newtheorem{lem}[theo]{Lemma}

\newtheorem{prop}[theo]{Proposition}
\newtheorem{conj}[theo]{Conjecture}

\begin{document}

\title{On problems about judicious bipartitions of graphs}

\author{Yuliang Ji\thanks{School of Mathematical Sciences, University of Science and Technology of China,
Hefei, Anhui 230026, China. Email: jiyl@mail.ustc.edu.cn.}
~~~~~
Jie Ma\thanks{School of Mathematical Sciences, University of Science and Technology of China,
Hefei, Anhui 230026, China. Email: jiema@ustc.edu.cn. Partially supported by NSFC projects 11501539 and 11622110.}
~~~~~
Juan Yan\thanks{College of Mathematics and Systems Science, Xinjiang University,
Urumqi, Xinjiang 830046, China. Email: yanjuan207@163.com. Partially supported by NSFC project 11501486.}
~~~~~
Xingxing Yu\thanks{School of Mathematics, Georgia Institute of Technology, Atlanta, GA 30332, USA. Email: yu@math.gatech.edu.
Partially supported by NSF grants DMS--1265564 and DMS-1600387.}
}

\date{}
\maketitle {\flushleft\large\bf Abstract:}
Bollob\'{a}s and Scott \cite{1} conjectured that every graph $G$ has a balanced bipartite spanning subgraph
$H$ such that for each $v\in V(G)$, $d_H(v)\ge (d_G(v)-1)/2$. In this
paper, we show that every graphic sequence has a realization for which this Bollob\'{a}s-Scott conjecture holds,
confirming a conjecture of Hartke and Seacrest \cite{HS}.
On the other hand, we give an infinite family of counterexamples to this Bollob\'{a}s-Scott conjecture, which
indicates that $\lfloor (d_G(v)-1)/2\rfloor$ (rather than
$(d_G(v)-1)/2$) is probably the correct lower bound.
We also study bipartitions $V_1, V_2$ of graphs with a fixed number of edges.
We provide a (best possible) upper bound on
$e(V_1)^{\lambda}+e(V_2)^{\lambda}$ for any real $\lambda\geq 1$
(the case $\lambda=2$ is a question of Scott \cite{Sc05}) and
answer a question of Scott \cite{Sc05} on
$\max\{e(V_1),e(V_2)\}$.

\begin{flushleft}
\textbf{Keywords:} bipartition; bisection; degree sequence; complete
$k$-partite graph; $\ell_\lambda$-norm

\end{flushleft}
\textbf{AMS classification}: 05C07, 05C70

\newpage

\section{Introduction}
For any positive integer $k$, let $[k]:=\{1,\ldots, k\}$.
Let $G$ be a graph and $V_1,\ldots, V_k$ be a
partition of $V(G)$. When $k=2$, such a partition is said to be a {\it
bipartition} of $G$.
A subgraph $H$ of a graph $G$ is said to be a  {\it bisection} of $G$ if $H$ is a bipartite spanning subgraph of $G$ and the partition sets of $H$
differ in size by at most one.
For $i,j\in [k]$, we use $e(V_i)$ to denote
the number of edges of $G$ with both ends in $V_i$ and use $e(V_i,
V_j)$ to denote the number of edges between $V_i$ and $V_j$.
\emph{Judicious partitioning problems} for graphs ask for partitions of graphs
that  bound a number of quantities simultaneously, such as
all $e(V_i)$ and $e(V_i, V_j)$. There has been extensive research on this type of problems over the past two decades.

As an attempt to better understand how edges of a graph are
distributed,
we study several judicious bipartitioning problems.  Specifically,  we study a conjecture of Bollob\'{a}s and Scott \cite{1} and its
degree sequence version conjectured by Hartke and Seacrest
\cite{HS}. We also study
two questions of Scott \cite{Sc05} on bipartitions $V_1,V_2$ of a
graph with $m$ edges, bounding $e(V_1)^2+e(V_2)^2$ and
$\max\{e(V_1),e(V_2)\}$ in terms of $m$.

For a graph $G$ and for any $v\in V(G)$, we use $d_G(v)$ to denote the
degree of the vertex $v$ in $G$.
It is well known that if $H$ is a maximum bipartite spanning subgraph of  a graph $G$, then
$d_H(v)\ge d_G(v)/2$ for each $v\in V(G)$.
This, however, may not be true if one requires $H$ to be a bisection, as observed by Bollob\'{a}s and Scott \cite{1} by considering the complete bipartite graphs $K_{2\ell+1,m}$ for $m\ge 2\ell+3$.
In an attempt to obtain a similar result for  bisections, Bollob\'{a}s and Scott \cite{1} conjectured that every graph $G$ has a bisection $H$ such that
\begin{equation}\label{conj:BS}
d_H(v)\ge (d_G(v)-1)/2 \text{~~~ for all~}v\in V(G).
\end{equation}
This conjecture for regular graphs was made by H\"aggkvist \cite{2} in
1978, and variations of this problem were studied by Ban and Linial  \cite{BL}.

Hartke and Seacrest~\cite{HS} studied a degree sequence version of this Bollob\'{a}s-Scott conjecture.
A nondecreasing sequence $\pi$ (of nonnegative integers) is said to be
{\it graphic} if it is the degree sequence of some finite simple graph $G$; and
such $G$ is called a $realization$ of the sequence $\pi$.
Hartke and Seacrest \cite{HS} proved that for any graphic sequence $\pi$ with even length,
$\pi$ has a realization $G$ which admits a bisection $H$ such that for all
$v\in V(G)$,  $d_H(v)\ge \lfloor (d_G(v)-1)/2\rfloor$.
They further conjectured that  for any graphic sequence $\pi$ with even length,
$\pi$ has a realization $G$ for which (\ref{conj:BS}) holds.
We prove this Hartke-Seacrest conjecture  for all graphic sequences.

For a graph $G$ and a labeling of its vertices $V(G)=\{v_1,\dots,
v_n\}$, we define the {\it parity bisection} of $G$ to be the
bisection with partition sets  $V_1$ and $V_2$,
where $V_i=\{v_j\in V(G): j\equiv i \mod 2\}$ for each $i\in [2]$, and
$E(H)=\{uv\in E(G): u\in V_1 \mbox{ and } v\in V_2\}$.

\begin{theo}\label{thm:HS}
Let $\pi=(d_1,\dots,d_n)$ be any graphic sequence with $d_1\ge \cdots \ge d_n$.
Then there exists a realization $G$ of $\pi$ with $V(G)=\{v_1,\dots, v_n\}$
and $d_G(v_i)=d_i$ for $i\in [n]$, such that if $H$ denotes the parity
bisection of $G$ then
$d_H(v_i)\ge (d_G(v_i)-1)/2$ for $i\in [n]$.
\end{theo}

\noindent The bound in Theorem \ref{thm:HS} is best possible, as shown
by the following example given by
Hartke and Seacrest~\cite{HS}.
Let $G$ be the join of a clique $K$ on $k$ vertices and an independent
set $I$ on $n-k$ vertices, where $n$ is even and $k<n/2$ is odd.
It is not hard to show that $G$ in fact is the unique realization of
the sequence $\pi=(d_1,\dots,d_n)$ with $d_1=\cdots=d_k=n-1$ and $d_{k+1}=\cdots=d_n=k$.
Let $H$ be an arbitrary bisection of $G$ with parts $A, B$ and, without loss of generality,
assume that $|A\cap V(K)|\le k/2$.  Since $k<n/2$, there must exist a
vertex $v\in B\cap I$.  So $d_H(v)=|A\cap V(K)|\le \lfloor
k/2\rfloor$. Since $d_G(v)=k$ and $k$ is odd, we see that 
$d_H(v)\le (d_G(v)-1)/2$.

\medskip


The second result in this paper gives indication  that perhaps the
lower bound in the original
Bollob\'{a}s-Scott conjecture was meant to be
$\lfloor(d_G(v)-1)/2\rfloor$ (rather than $(d_G(v)-1)/2$).

\begin{prop}\label{prop:BS-3partite}
Let $r_1,r_2,r_3$ be pairwise distinct odd integers such that for every $i\in [3]$,
$r_i\notin\left\{1, \lfloor(r_1+r_2+r_3)/2\rfloor, \lceil(r_1+r_2+r_3)/2\rceil\right\}$.
Then for any bisection $H$ of the complete 3-partite graph
$G:=K_{r_1,r_2,r_3}$, there always exists a vertex $v$ with
$d_H(v)<(d_G(v)-1)/2$.
\end{prop}

\noindent This result will follow  from a more general result, Proposition~\ref{t2}, on all
complete multipartite graphs. We remark here that, for each complete multipartite graph $G$, it is
easy (as we will see in Section 3) to find a bisection $H$ of $G$
such that 
$d_H(v)\ge \lfloor (d_G(v)-1)/2\rfloor$ for all $v\in V(G)$.
However, for general graphs, even the following weaker version of the
Bolloba\'{a}s-Scott conjecture seems
quite difficult to prove (or disprove).

\begin{conj}
There exists some absolute constant $c>0$ such that every graph $G$
has a bisection $H$ with $d_H(v)\ge (d_G(v)-c)/2$  for all $v\in V(G)$.
\end{conj}

We now turn our discussion to problems on general  bipartitions.
Answering a question of Erd\H{o}s, Edwards \cite{ed} showed in 1973
that every graph with $m$ edges admits a bipartition $V_1,V_2$ such that
$e(V_1,V_2)\ge m/2+t(m)/2$,  where
$$t(m):=\sqrt{m/2+1/16}-1/4.$$
This bound is best possible for the complete graphs of odd order.
Bollob\'{a}s and Scott \cite{bs99} extended Edwards' bound by showing that
every graph $G$ with $m$ edges has a bipartition $V_1, V_2$ simultaneously satisfying
$e(V_1,V_2)\ge m/2+t(m)/2$ and
$\max\{e(V_1),e(V_2)\}\le m/4+t(m)/4$,
where both bounds are tight for the complete graphs of odd order.

Scott \cite{Sc05} provided an interesting viewpoint by introducing
norm for partitions.
For a real number $\lambda>0$ and a bipartition $V_1, V_2$ of a graph $G$, define {\it the $\ell_\lambda$-norm} of
$(V_1,V_2)$ to be $\left(e(V_1)^\lambda+e(V_2)^\lambda\right)^{1/\lambda}$.
Then to maximize $e(V_1,V_2)$ is equivalent to minimize the $\ell_1$-norm of $(V_1,V_2)$,
while minimizing $\max\{e(V_1),e(V_2)\}$ is the same as  minimizing the $\ell_\infty$-norm of $(V_1,V_2)$.
It is natural to consider other norms. In particular,
Scott asked for the maximum of
$$\min_{V(G)=V_1\cup V_2} e(V_1)^2+e(V_2)^2$$ over graphs $G$ with $m$
edges, see Problem 3.18 in \cite{Sc05}.
We provide an answer to this question by proving the following general result.
\begin{theo}\label{thm:norm}
Let $m$ be any positive integer and $\lambda\geq 1$ be any real number.
Then, for any graph $G$ with $m$ edges,
$$\min_{V(G)=V_1\cup V_2} e(V_1)^\lambda+e(V_2)^\lambda\le {t(m) \choose 2}^\lambda+{t(m)+1 \choose 2}^\lambda.$$
Moreover, the equality holds if and only if $G$ is a complete graph of odd order.
\end{theo}

\noindent We also consider analogous questions for $k$-partitions in Section 4.

\medskip

Though Edward's bound is tight for all integers $m=\binom{n}{2}$, Erd\H{o}s \cite{Erd95} conjectured that
the difference between Edwards' bound and the truth can still be
arbitrarily large for other $m$.
This was confirmed by Alon \cite{alon}: every graph with $m=n^2/2$ edges
admits a bipartition $V_1,V_2$ such that $e(V_1,V_2)\ge
m/2+t(m)/2+\Omega(m^{1/4})$. Bollob\'{a}s and Scott \cite{1, Sc05} made a similar conjecture
for $\max\{e(V_1),e(V_2)\}$: for certain  $m$,   $\max
\{e(V_1),e(V_2)\}$ can be arbitrary far from $m/4+t(m)/4$.  Ma and Yu
\cite{MY16} proved that every graph with $m=n^2/2$ edges
admits a bipartition $V_1,V_2$ such that $\max\{e(V_1),e(V_2)\}\le  m/4+t(m)/4-\Omega(m^{1/4})$.
Another result in the same spirit was given by Hofmeister and Lefmann \cite{HL} that any graph with $\binom{kn}{2}$ edges
has a $k$-partition $V_1,...,V_k$ with $\sum_{i=1}^k e(V_i)\le k\binom{n}{2}$,
which beats the trivial upper bound $\frac{1}{k}\binom{nk}{2}$.

Motivated by these results, Scott asked the following question: does every graph $G$ with
$\binom{kn}{2}$ edges have a vertex partition into $k$ sets, each of
which contains at most $\binom{n}{2}$ edges? ( See Problem 3.9 in \cite{Sc05}.)
We show that the answer to this question is negative for $k=2$.
\begin{theo}\label{thm:2n}
There exist infinitely many positive integers $n$ and for each such $n$ there is a graph
$G$ with $\binom{2n}{2}$ edges, such that, for every bipartition
$V_1, V_2$ of $G$, $\max\{e(V_1),e(V_2)\}\ge \binom{n}{2}+5n/48.$
\end{theo}

\medskip

This paper is organized as follows. We prove Theorem~\ref{thm:HS} in Section 2,
and then investigate complete multipartite graphs for
the Bollob\'{a}s-Scott conjecture in Section 3.
In Section 4, we discuss the questions of Scott and complete the proofs of Theorems \ref{thm:norm} and \ref{thm:2n}.

\section{Hartke-Seacrest conjecture}
In this section, we prove Theorem \ref{thm:HS}.
We need two operations on a sequence of non-negative integers.
Let $\pi=(d_1,\dots,d_n)$ with $d_1\geq\cdots\geq d_n$.
By removing $d_i$ from $\pi$ and subtracting $1$ from the
$d_i$ remaining elements of $\pi$ with lowest indices, we obtain a new sequence
$\pi'=(d'_1,\dots, d'_{i-1},d'_{i+1},\dots,d'_n)$, and we say that $\pi'$ is obtained from $\pi$ by
{\it laying off}  $d_i$.  This operation was introduced by Kleitman and Wang \cite{6},
and they proved the following.

\begin{lem}[Kleitman-Wang \cite{6}]
For any $i\in [n]$, the sequence $\pi=(d_1,\dots,d_n)$ with $d_1\ge
\ldots \ge d_n$ is graphic if and only if
the sequence $\pi'$ obtained from $\pi$ by laying off $d_i$ is graphic.
\end{lem}

It is easy to see that the sequence $\pi'$ obtained from $\pi$  by
laying off $d_i$ need not be non-increasing.
To avoid this issue, Hartke and Seacrest \cite{HS} introduced a
variation of the above operation.
Choose a fixed $i\in [n]$.
Let $s$ be the smallest value among the $d_i$
largest elements in $\pi$, not including the $i$th element of $\pi$
(namely, $d_i$ itself). Let $S=\{j\in
[n]-\{i\}:  d_j>s\}$. Note that $|S|<d_i$. Let $T$ be the
set of $d_i-|S|$ largest indices $j$ with $j\not=i$ and $d_j=s$.
Then by {\it laying off $d_i$ with order}, we remove $d_i$ from $\pi$
and subtract $1$ from $d_j$ for all $j\in S\cup T$. If
$\pi'=(d'_1,\dots d'_{i-1},d'_{i+1},\dots,d'_n)$ denotes the new
sequence, then it has the monotone property $d'_1\geq\cdots\geq d'_n$.

Clearly, the sequence obtained from $\pi$  by laying off $d_i$ with order is just a permutation of the
sequence obtained from $\pi$ by laying off $d_i$. So the following is true.

\begin{lem}[Hartke and Seacrest \cite{HS}]\label{lem:HS}
For any $i\in [n]$, the sequence $\pi=(d_1,\dots,d_n)$ with $d_1\ge
\ldots \ge d_n$ is graphic if and only if
the sequence obtained from $\pi$ by laying off $d_i$ with order is graphic.
\end{lem}

We give a brief outline of our proof of Theorem~\ref{thm:HS}. We
choose two consecutive elements $d_\ell$ and $d_{\ell+1}$ of $\pi$.
Using Lemma \ref{lem:HS} we obtain a new graphic sequence $\pi''$ of
length $n-2$ by first laying off $d_{\ell+1}$ with order and then laying
off $d_{\ell}$ with order. By induction, $\pi''$ has an $(n-2)$-vertex
realization $F$  whose parity bisection $J$ has the desired
property. We then show that one can form $G$ from $F$ by adding two new vertices (for $d_\ell$ and $d_{\ell+1}$)
and choosing their neighbors, so that the parity bisection of $G$
satisfies Theorem~\ref{thm:HS}.

\begin{proof}[Proof of Theorem \ref{thm:HS}]
We apply induction on the length $n$ of the graphic
sequence $\pi=(d_1,\ldots, d_n)$ with $d_1\ge \ldots \ge d_n$.
The assertion is trivial when $n = 1,2$.
So we may assume that $n\ge 3$ and  the assertion holds for all graphic sequences with length less than $n$.
Then there exist two consecutive elements of $\pi$ that are identical;
so let $\ell \in [n-1]$ be fixed such that $$d_\ell = d_{\ell+1} = k.$$


Let $\pi'=(d_1',...,d_\ell',d_{\ell+2}',...,d_{n}')$ be the sequence obtained
from $\pi$ by laying off $d_{\ell+1}$ with order. Let
$\pi''=(d_1'',...,d_{\ell-1}'',d_{\ell+2}'',...,d_{n}'')$ be the sequence
obtained from $\pi'$ by laying off $d_l'$ with order.
By Lemma \ref{lem:HS}, $\pi'$ and $\pi''$ both are graphic sequences.

Let $\omega=(f_1, ...,f_{n-2})$ be the sequence obtained from
$\pi$ with $d_\ell$ and $d_{\ell+1}$ removed,
and re-indexed so that the indices are consecutive, i.e., $f_i=d_i$
for $i\in [\ell-1]$ and $f_i=d_{i+2}$ for $i\in [n-2]\setminus [\ell-1]$.
Let $\omega'=(f_1', ...,f_{n-2}')$ be the sequence obtained from $\pi'$
with $d'_\ell$ removed, and re-indexed so that the indices are consecutive. Also,
let $\omega''=(f_1'' , ...,f_{n-2}'')$ be the sequence obtained from
$\pi''$ by re-indexing so that the indices are consecutive.
Note that  $\omega''$ is a  graphic sequence.

To turn a realization of $\omega''$ to a realization of $\pi$,
we need to track the changes between $f_i$ and $f_i''$ for all $i\in
[n-2]$. Note that  $0\le f_i-f_i''\le 2$.
Let $$X_{1}=\{i\in [n-2]: f''_i=f_i-1\}, ~~~  X_{2}=\{i\in [n-2]: f''_i=f_i-2\}$$
and
\begin{equation*}
K=d_{\ell}'=\left\{\begin{array}{ll}
    k-1 & \text{ if } d'_{\ell} =d_{\ell}-1, \\
    k & \text{ if } d'_{\ell} =d_{\ell}.
  \end{array}\right.
\end{equation*}
So $K=\sum_{i\in [n-2]}|f_i-f_i'|=\sum_{i\in [n-2]}|f_i'-f_i''|$; hence
\begin{equation}\label{equ:2K}
|X_1|+2|X_2|=2K.
\end{equation}

We now prove two claims asserting certain properties on $X_1$ and
$X_2$. For convenience, we introduce some notation.
For nonempty sets $A$ and $B$ of integers, we write $A<B$ if the maximum integer in $A$ is less than the minimum integer in $B$.
A set  $S$ of integers is {\it consecutive} if it consists of consecutive
integers. A sequence of pairwise disjoint sets, $A_1,...,A_t$, of integers is said to be {\it
  consecutive} if $A_1\cup...\cup A_t$ is consecutive and, for any $i,j\in [t]$ with $i<j$ and $A_i$ and $A_j$
nonempty, we have $A_i<A_j$.

\bigskip

{\noindent \bf Claim 1.} There exist consecutive sets $R_1,R_2,R_1',R_2',Q$ such
that  $X_{1}=R_1'\cup R_2'$ and $X_2=R_1\cup R_2$
such that
\begin{itemize}
\item [(a)] the sequence $R_1, R_1', Q, R_2'$ is consecutive,
\item [(b)] either $R_2=\emptyset$ or $R_2=Q$, and
\item [(c)] $f_i''=f_j''+1$ for all $i\in R_1', j\in R_2'$.
\end{itemize}

\begin{proof}[Proof of Claim 1.]
Let $s$ be the minimum of the largest $K$ numbers in $\omega=(f_1,...,f_{n-2})$.
(Note that this $s$ is the same as the $s$ in the definition of laying off $d_{\ell+1}$ with order from $\pi$.)
 In order to keep
track whether  $f_i'=f_i$ or $f_i'=f_i-1$ and whether $f_i''=f_i'$ or
$f_i''=f_i'-1$, we divide $[n-2]$ into  six pairwise disjoint sets:
\begin{equation*}
\begin{array}{ll}
    A=\{i\in [n-2]: f_i \geqslant s+2\}, ~~~~~&  D=\{i\in [n-2]: f_i=s, f_i'=f_i-1\}, \\
    B=\{i\in [n-2]: f_i=s+1\},~~~~~&  E=\{i\in [n-2]: f_i=s-1\},\\
    C=\{i\in [n-2]: f_i=s, f_i'=f_i\}, ~~~~~& F=\{i\in [n-2]: f_i \leqslant s-2\}.
\end{array}
\end{equation*}

By the definitions of $\pi'$ and $\omega'$, we see that  $A,B,C,D,E,F$ is consecutive and
\begin{equation*}
\begin{array}{ll}
\forall \ i \in A, & ~f_i'=f_i-1 \geqslant s+1,\\
\forall \ i \in B, & ~f_i'=f_i-1=s,\\
\forall \ i \in C, & ~f_i'=f_i =s,\\
\forall \ i \in D, & ~f_i'=f_i-1=s-1,\\
\forall \ i \in E, & ~f_i'=f_i =s-1,\\
\forall \ i \in F, & ~f_i'=f_i \leqslant s-2.
\end{array}
\end{equation*}
Thus, it is easy to see that $A\cup B\cup D=\{i\in [n-2]:
f'_i=f_i-1\}$; so $|A|+|B|+|D|=K.$

Let $Y=\{i\in [n-2]: f''_i=f_i'-1\}$. Then it follows that $$A\subseteq Y \text{~~and ~~} |Y|=K=|A|+|B|+|D|.$$
To complete our proof of  Claim 1, we distinguish four cases based on
relations among the sizes of $B,C,D,E$.

First, suppose $|C|\geq |B|+|D|$. Let $C''$ consist of the last $|B|+|D|$ integers in $C$, and $C':=C\setminus  C''$.
Then we see that $Y=A\cup C''$. Let $R_1=A$, $R_2=\emptyset$, $R_1'=B$,
$R_2'=C''\cup D$ and $Q=C'$.
It is easy to check that  $X_{1}=R_1'\cup R_2'$ and $X_2=R_1\cup R_2$,
and (a) and (b) holds.  Note that $f_i''=s$ for $i\in R_1'$, and
$f_j''=s-1$ for $j\in R_2'$; so (c) holds.

Next, suppose $|D|\leq|C|<|B|+|D|$. Let $B''$ consist of the last $|B|+|D|-|C|$ integers in $B$, and $B'=B\setminus B''$.
Then $Y=A\cup B''\cup C$. Let $R_1=A$, $R_2=Q= B''$, $R_1'=B'$ and
$R_2'=C\cup D$.
It is easy to check that  $X_{1}=R_1'\cup R_2'$ and $X_2=R_1\cup R_2$,
and  that (a) and (b) holds.
Note that $f_i''=s$ for $i\in R_1'$, and $f_j''=s-1$ for $j\in R_2'$;
so (c) holds.

Now assume $|C|<|D|\leq |C|+|E|$.
Let $E''$ consist of the last $|D|-|C|$ integers in $E$, and $E'=E\setminus E''$.
Then $Y=A\cup B\cup C\cup E''$. Let $R_1=A\cup B$, $R_2= \emptyset$,
$R_1'=C\cup D$, $R_2'=E''$, and $Q=E'$.
It is easy to check that  $X_{1}=R_1'\cup R_2'$ and $X_2=R_1\cup R_2$,
and (a) and (b) holds.
Note that  $f_i''=s-1$ for $i\in R_1'$ and  $f_j''=s-2$ for $j\in
R_2'$; so (c) holds.

Finally we consider the case when  $|D|>|C|+|E|$.
Let $D''$ consist of the last $|D|-|C|-|E|$ integers in $D$, and $D'=D\setminus D''$.
Then  $Y=A\cup B\cup C\cup D''\cup E$.
Let $R_1=A\cup B$, $R_2=Q=D''$, $R_1'=C\cup D'$ and $R_2'=E$. It is easy to check that  $X_{1}=R_1'\cup R_2'$ and $X_2=R_1\cup R_2$,
and (a) and (b) holds.
Note that $f_i''=s-1$ for $i\in R_1'$ and  $f_j''=s-2$ for $j\in
R_2'$; so (c) holds.
\end{proof}

Let $I_1=\{i\in [n-2]: i\equiv 1 \mod 2\}$ and $I_2=\{i\in [n-2]: i\equiv 0 \mod 2\}$.

\bigskip

{\noindent \bf Claim 2.} $|X_1\cap I_1|-|X_1\cap I_2|\in \{-2,0,
2\}$. Moreover,  $|X_1\cap I_1|-|X_1\cap I_2|=0$ implies
$|X_2\cap I_1|-|X_2\cap I_2|\in \{-1,0,1\}.$

\begin{proof}[Proof of Claim 2.]
By \eqref{equ:2K}, we see $|X_1|$ must be even. So $|R'_1|$ and $|R'_2|$ are of the same parity.
Since both $R_1'$ and $R_2'$ are consecutive, $|X_1\cap I_1|-|X_1\cap I_2|\in \{-2,0, 2\}$.

Now suppose $|X_1\cap I_1|-|X_1\cap I_2|=0$.
If $R_2=\emptyset$, then $X_2=R_1$ is a consecutive set and thus
$|X_2\cap I_1|$ and $|X_2\cap I_2|$ differ by at most one. So we may
assume $R_2\ne \emptyset$. Then $R_2=Q$ by Claim 1. As the sequence $R_1,R_1',Q, R_2'$ is consecutive,
we see that $X_1\cup X_2$ is a consecutive set; so $|(X_1\cup X_2)\cap I_1|$ and $|(X_1\cup X_2)\cap I_2|$ differ by at most one.
Hence, since $|X_1\cap I_1|-|X_1\cap I_2|=0$,  $||X_2\cap I_1|-|X_2\cap I_2||\le 1$.
\end{proof}

We are ready to construct a realization of $\pi=(d_1,...,d_{n})$.
Recall that $\omega''=(f_1'',...,f_{n-2}'')$ is a graphic sequence.
By induction hypothesis, there exists a realization $F$ of $\omega''$
with $V(F)=\{w_1,...,w_{n-2}\}$ and $d_{F}(w_i)=f_i''$ for $i\in
[n-2]$, such that the parity bisection $J$ of $F$ satisfies
\begin{equation}\label{equ:G-rho(H)}
 d_{J}(w_i)\ge (d_F(w_i)-1)/2 \text{~~ for all }i\in [n-2].
\end{equation}
Let $W_j=\{w_i: i\equiv j \mod 2\}$ for $j\in [2]$.

In what follows, we will construct a graph $G$ as the realization of
$\pi$ such that its parity bisection $H$ of $G$ satisfies
$d_{H}(v)\ge (d_G(v)-1)/2$ for all $v\in V(G)$,
by adding two new vertices $a,b$ (so $V(G)=V(F)\cup \{a,b\}$) and some
edges from these two vertices to $F$ (which we will describe in three
separate cases).
Notice that if $K=k-1$, then we would add the edge $ab$ as well; so
for convenience, let
\begin{equation*}
\epsilon=\left\{\begin{array}{ll}
    1, & \text{ if } K=k-1, \\
    0, & \text{ if } K=k.
  \end{array}\right.
\end{equation*}
We write $V(G)=\{v_1,...,v_{n}\}$ such that
$v_i=w_i$ for $i<\ell$, $\{v_{\ell}, v_{\ell+1}\}=\{a,b\}$, and
$v_i=w_{i-2}$ for $\ell+1<i\le n$.

In view of Claim 2, we consider the following three cases. In each of
these three cases, we use $a$ to represent the vertex in $\{v_{\ell}, v_{\ell+1}\}$ with odd index.
So the parity partition of $V(G)$ is  $$V_1=W_1\cup \{a\} \text{~~ and~~} V_2=W_2\cup \{b\}.$$

\bigskip

Case $1$. $|X_1\cap I_1|-|X_1\cap I_2|=0$.

\medskip

We know $F\subseteq G$ and $V(G)=V(F)\cup \{a,b\}$, and we need to add edges at $a$ and
$b$ to form $G$, a realization of $\pi$. Add
$ab$ if $\epsilon=1$, $av_i$ for all $i\in X_2\cup (X_1\cap I_2)$, and $bv_j$ for all $j\in X_2\cup (X_1\cap I_1)$.
Since $|X_1\cap I_1|=|X_1\cap I_2|$, $G$ is a realization of $\pi$.
Let $H$ denote the parity bisection of $G$; so $V_1,V_2$ are the
partition sets of $H$. We need to show that  
$d_{H}(v)\ge (d_G(v)-1)/2$ for all $v\in V(G)$.

For each $w_i$ with $i\notin X_1\cup X_2$, its
neighborhoods in $F,G$ are the same; so  by \eqref{equ:G-rho(H)},
$d_{H}(w_i)=d_{J}(w_i)\ge (d_F(w_i)-1)/2=(d_G(w_i)-1)/2$.

For vertices $w_i$ with $i\in X_2$, we have
$d_{G}(w_i)=d_{F}(w_i)+2$ and $d_{H}(w_i)=d_{J}(w_i)+1$;
so by \eqref{equ:G-rho(H)}, 
$d_{H}(w_i)=d_J(w_i)+1\ge (d_F(w_i)-1)/2+1 =(d_G(w_i)-1)/2$.

For vertices $w_i$ with $i\in X_1$, we have
$d_{G}(w_i)=d_{F}(w_i)+1$ and $d_{H}(w_i)=d_{J}(w_i)+1$; so
by \eqref{equ:G-rho(H)}, 
$d_H(w_i)=d_J(w_i)+1\ge (d_F(w_i)-1)/2+1 >(d_G(w_i)-1)/2$.

For the vertex $a$, we have $d_{G}(a)=|X_2|+|X_1\cap I_2|+\epsilon$ and $d_{H}(a)=|X_2\cap I_2|+|X_1\cap I_2|+\epsilon$.
Note that in this case, by Claim 2, we have $||X_2\cap I_1|-|X_2\cap I_2||\le 1$,
which implies that 
$$2d_{H}(a)-d_{G}(a)=\left(|X_2\cap I_2|-|X_2\cap I_1|\right)+|X_1\cap I_2|+\epsilon\ge -1.$$
Hence, $d_H(a)\ge (d_G(a)-1)/2$.

Similarly, for the vertex $b$, we have $d_{G}(b)=|X_2|+|X_1\cap I_1|+\epsilon$ and $d_{H}(b)=|X_2\cap I_1|+|X_1\cap I_1|+\epsilon$.
Note, by Claim 2,  $||X_2\cap I_1|-|X_2\cap I_2||\le 1$; so
$$2d_{H}(b)-d_{G}(b)=\left(|X_2\cap I_1|-|X_2\cap I_2|\right)+|X_1\cap I_1|+\epsilon\ge -1.$$
Hence, $d_H(b)\ge (d_G(b)-1)/2$.

\bigskip

Case $2$. $|X_1\cap I_2|-|X_1\cap I_1|=2$.

\medskip

Recall that $X_1=R'_1\cup R'_2$, where each $R'_i$ is consecutive.
Thus it follows that $|R'_i\cap I_2|=|R'_i\cap I_1|+1$ for $i\in [2]$.
Since the sequence $R_1, R'_1, Q, R'_2$ is consecutive and starts from the integer $1$,
we see that $|R_1\cap I_2|= |R_1\cap I_1|-1$ and $|Q\cap I_2|= |Q\cap
I_1|-1$. Therefore, since $R_2=\emptyset$ or $R_2=Q$ (by (b) of Claim 1), we
have
\begin{equation}\label{equ:Case2}
-2\le |X_2\cap I_2|-|X_2\cap I_1|\le -1.
\end{equation}

We claim that there exists some $z\in X_1\cap I_2$ with
$d_J(w_z)\ge d_F(w_z)/2$.
To see this, choose  $x\in R'_1\cap I_2$ and $y\in R'_2\cap I_2$.
By \eqref{equ:G-rho(H)} ,
we have 
$d_J(w_x)\ge (d_F(w_x)-1)/2$ and $d_J(w_y)\ge (d_F(w_y)-1)/2$. By (c)
of Claim 1,  $d_F(w_x)=d_F(w_y)+1$.
Observe that for any vertex $u$ of $F$, $d_F(u)$ and 
$2d_J(u)-d_F(u)$ are of the same parity;
so $2d_J(w_x)-d_F(w_x)$ and $2d_J(w_y)-d_F(w_y)$ must have different parities.
Therefore there exists $z\in\{x,y\}$ such that 
$d_J(w_z)\ge d_F(w_z)/2$, proving the claim.

We now add edges at $a$ and
$b$ to form $G$ from $F$: add $ab$ if $\epsilon=1$, $av_i$ for all $i\in X_2\cup (X_1\cap I_2)\setminus\{z\}$,
and $bv_j$ for all $j\in X_2\cup (X_1\cap I_1)\cup \{z\}$.
Since $|X_1\cap I_2|=|X_1\cap I_1|+2$, $G$ is a realization of $\pi$. Next we show that the parity
bisection $H$ of $G$ satisfies the property that $d_H(v)\ge
(d_G(v)-1)/2$ for all $v\in V(G)$.

For each $w_i$ with $i\notin X_1\cup X_2$, its
neighborhoods in $F,G$ are the same; so by \eqref{equ:G-rho(H)},
$d_H(w_i)=d_J(w_i)\ge (d_F(w_i)-1)/2=(d_G(w_i)-1)/2$.

For each $w_i$ with $i\in X_2$, $d_{G}(w_i)=d_F(w_i)+2$ and $d_{H}(w_i)=d_J(w_i)+1$.
Hence by \eqref{equ:G-rho(H)} and the way we choose $z$, $d_H(w_i)=d_J(w_i)+1\ge (d_F(w_i)-1)/2+1=(d_G(w_i)-1)/2$.

For  $w_i$ with $i\in X_1\setminus \{z\}$,
we have $d_{G}(w_i)=d_F(w_i)+1$ and $d_{H}(w_i)=d_J(w_i)+1$;
so by \eqref{equ:G-rho(H)}, 
$d_H(w_i)=d_J(w_i)+1\ge (d_F(w_i)-1)/2+1>(d_G(w_i)-1)/2$.

The vertex $w_z$ satisfies
$d_{G}(w_z)=d_F(w_z)+1$ and $d_{H}(w_z)=d_J(w_z)$.
Hence by \eqref{equ:G-rho(H)}, $d_H(w_z)=d_J(w_z)\ge d_F(w_z)/2=(d_G(w_z)-1)/2$.

For the vertex $a$, by definition we have
$d_{G}(a)=|X_2|+|X_1\cap I_2|-1+\epsilon$ and $d_{H}(a)=|X_2\cap I_2|+|X_1\cap I_2|-1+\epsilon$.
By \eqref{equ:Case2} and the fact that $|X_1\cap I_2|\ge 2$,
$$2d_{H}(a)-d_G(a)=(|X_2\cap I_2|-|X_2\cap I_1|)+|X_1\cap I_2|-1+\epsilon\ge -1.$$
Hence, $d_H(a)\ge (d_G(a)-1)/2$.

For the vertex $b$, we have
$d_{G}(b)=|X_2|+|X_1\cap I_1|+1+\epsilon$ and $d_{H}(b)=|X_2\cap I_1|+|X_1\cap I_1|+\epsilon$.
This, together with \eqref{equ:Case2}, imply that
$$2d_{H}(b)-d_G(b)=(|X_2\cap I_1|-|X_2\cap I_2|)+|X_1\cap I_1|-1+\epsilon\ge 0.$$
Hence,  $d_H(b)\ge (d_G(b)-1)/2$.

\bigskip

Case $3$. $|X_1\cap I_1|-|X_1\cap I_2|=2$.

\medskip

In this case, we have $|R'_i\cap I_1|=|R'_i\cap I_2|+1$ for $i\in [2]$
(as $R_1'$ and $R_2'$ are consecutive).
Because $R_1, R_1', Q, R_2'$ is consecutive, it follows that
$|R_1\cap I_1|=|R_1\cap I_2|$ and $|Q\cap I_1|=|Q\cap I_2|-1$. Since
$R_2=\emptyset$ or $R_2=Q$ (by (b) of Claim 1),
\begin{equation}\label{equ:Case3}
0\le |X_2\cap I_2|-|X_2\cap I_1|\le 1.
\end{equation}

Since $|X_1|$ is even and $|R'_i\cap I_1|=|R'_i\cap I_2|+1$ for $i\in [2]$, there exist $x\in R_1'\cap I_1$ and $y\in
R_2'\cap I_1$. By \eqref{equ:G-rho(H)} and (c) of Claim 1,
we have 
$d_J(w_x)\ge (d_F(w_x)-1)/2$, $d_J(w_y)\ge (d_F(w_y)-1)/2$,
and $d_F(w_x)=d_F(w_y)+1$.
Since for any vertex $u$ of $F$, $d_F(u)$ and 
$2d_J(u)-d_F(u)$ are of the same parity,
$2d_J(w_x)-d_F(w_x)$ and $2d_J(w_y)-d_F(w_y)$ must have different parities.
Therefore there exists $z\in\{x,y\}$ such that 
$d_J(w_z)\ge d_F(w_z)/2$.

We now add edges at $a$ and $b$ to form the graph $G$: add
$ab$ if $\epsilon=1$, $av_i$ for all $i\in X_2\cup (X_1\cap I_2)\cup\{z\}$,
and $bv_j$ for all $j\in X_2\cup (X_1\cap I_1)\setminus\{z\}$.
Since $|X_1\cap I_1|=|X_1\cap I_2|+2$, $G$ is a realization of $\pi$.
We need to verify that 
$d_H(v)\ge (d_G(v)-1)/2$ for all $v\in V(G)$.

If $v=w_i$ for some $i\notin X_1\cup X_2$, then $d_{G}(w_i)=d_F(w_i)$
and $d_{H}(w_i)=d_J(w_i)$; so 
by \eqref{equ:G-rho(H)}, $d_H(w_i)=d_J(w_i)\ge
(d_F(w_i)-1)/2=(d_G(w_i)-1)/2$.

If  $v=w_i$ for some $i\in X_2$, then $d_{G}(w_i)=d_F(w_i)+2$ and
$d_{H}(w_i)=d_J(w_i)+1$; again  
by \eqref{equ:G-rho(H)}, $d_H(w_i)=d_J(w_i)+1\ge (d_F(w_i)-1)/2+1=(d_G(w_i)-1)/2$.

If $v=w_i$ for some $i\in X_1\setminus \{z\}$,
then $d_{G}(w_i)=d_F(w_i)+1$ and $d_{H}(w_i)=d_J(w_i)+1$; so
by \eqref{equ:G-rho(H)}, $d_H(w_i)=d_J(w_i)+1\ge (d_F(w_i)-1)/2+1>(d_G(w_i)-1)/2$.

Now suppose  $v=w_z$. Note that $d_{G}(w_z)=d_F(w_z)+1$ and $d_{H}(w_z)=d_J(w_z)$.
So 
$d_H(w_z)=d_J(w_z)\ge d_F(w_z)/2=(d_G(w_z)-1)/2$.

Suppose $v=a$. Note that
$d_{G}(a)=|X_2|+|X_1\cap I_2|+1+\epsilon$ and $d_{H}(a)=|X_2\cap I_2|+|X_1\cap I_2|+\epsilon$.
By \eqref{equ:Case3}, 
$$2d_{H}(a)-d_G(a)=(|X_2\cap I_2|-|X_2\cap I_1|)+|X_1\cap I_2|-1+\epsilon\ge -1.$$
So $d_H(a)\ge (d_G(a)-1)/2$.

Finally, suppose $v=b$. We have
$d_{G}(b)=|X_2|+|X_1\cap I_1|-1+\epsilon$ and $d_{H}(b)=|X_2\cap I_1|+|X_1\cap I_1|-1+\epsilon$.
By \eqref{equ:Case3} and the fact that $|X_1\cap I_1|\ge 2$,
$$2d_{H}(b)-d_G(b)=(|X_2\cap I_1|-|X_2\cap I_2|)+|X_1\cap I_1|-1+\epsilon\ge 0.$$
So $d_H(b)\ge (d_G(b)-1)/2$.
\end{proof}

\section{Complete multipartite graphs}
For convenience, we say that a bisection $H$ of a graph $G$ is {\it good} if for each $v\in V(G)$, $2d_H(v)\ge d_G(v)-1$.
Thus the Bollob\'{a}s-Scott conjecture says that every graph contains a good bisection.
Here we discuss which complete multipartite graphs have good bisections.
Throughout the rest of this section, let $G:=K_{r_1,\dots,r_k}$, and let $X_1, \ldots, X_k$ denote the
partition sets of $G$ with $|X_i|=r_i$ for all $i\in [k]$.

First, we note that whenever $|V(G)|$ is even, $G$ has a good bisection.
Since $|V(G)|$ is even, $V(G)$ has a partition $V_1,V_2$ such that $|V_1|=|V_2|$, $||X_i\cap V_1|-|X_i\cap
V_2||=1$ if $|X_i|$ is odd, and $|X_i\cap V_1|=|X_i\cap V_2|$ if $|X_i|$
is even. Let $H$ denote the maximum bisection of $G$ with partition
sets $V_1$ and $V_2$.  For any $v\in V(G)$, $v\in X_i$ for some $i\in
[k]$. Note that $d_G(v)=|V(G)|-|X_i|$ and $d_H(v)\ge
(|V(G)|-|X_i|-1)/2= (d_G(v)-1)/2$.

We will see that this is not always the case when $|V(G)|$ is odd.
The main result of this section is a necessary and sufficient condition for a complete multipartite graph with odd order
to contain a good bisection. As a consequence, we show that for many complete multipartite graphs $G$,
$G$ (and even $G$ minus an edge) does not have a good bisection.
(However, it is not hard to see that such  $G$ does have a bisection
$H$ such that for each $v\in V(G)$, $d_H(v)\ge \lfloor (d_G(v)-1)/2 \rfloor$.)

For a bisection $H$ of $G$ with partition sets $V_1$ and $V_2$,
we say that $X_i$ {\it crosses} $H$ if $X_i\cap V_j\not=\emptyset$ for
$j\in [2]$.
 For a subset $W\subseteq V(G)$, let $\overline{W}=V(G)\setminus W$.
We need two easy lemmas.

\begin{lem}\label{l1}
Let  $G=K_{r_1,\dots,r_k}$ and let $X_i$,  $i\in [k]$, be the
partition sets of $G$. Suppose $H$ is a good bisection of $G$ with
partition sets $V_1$ and $V_2$. Then the following statements hold for  $i\in [k]$.

\begin{itemize}
\item [$(i)$]  If $X_i$ crosses $H$ and $|\overline{X_i}|$ is even then $|\overline{X_i}\cap V_1|=|\overline{X_i}\cap V_2|$ and
 $||X_i\cap V_1|-|X_i\cap V_2||\leq 1$.
\item [$(ii)$]  If $X_i$ crosses $H$ and $|\overline{X_i}|$ is odd then $||\overline{X_i}\cap V_1|-|\overline{X_i}\cap V_2||=1$ and
 $||X_i\cap V_1|-|X_i\cap V_2||\leq 2$.
\end{itemize}
\end{lem}
\begin{proof}
Suppose $X_i$ crosses $H$. Then there exist $v_j\in X_i\cap V_j$ for
$j\in [2]$. Note that  $d_G(v_1)=d_G(v_2)=|\overline{X_i}|$.
Thus, since $H$ is a good bisection, $d_H(v_j)=|\overline{X_i}\cap
V_{3-j}|\geq\lfloor |\overline{X_i}|/2\rfloor$ for $j\in [2]$.
Notice that $|\overline{X_i}\cap V_1|+|\overline{X_i}\cap V_2|=|\overline{X_i}|.$
So $||\overline{X_i}\cap V_1|-|\overline{X_i}\cap V_2||\leq
1$.

If $|\overline{X_i}|$ is even then $|\overline{X_i}\cap
V_1|=|\overline{X_i}\cap V_2|$, and $(i)$ holds. So assume
$|\overline{X_i}|$ is odd. Since  $||V_1|-|V_2||\leq 1$, $||X_i\cap V_1|-|X_i\cap
V_2||\leq 2$, and $(ii)$ holds.
\end{proof}

\begin{lem}\label{l2}
Let $G=K_{r_1,\dots,r_k}$ with $|V(G)|$ odd, let $X_1,\dots,X_k$ be the partition sets of $G$,
and let $H$ be a good bisection of $G$ with partition sets $V_1,V_2$ such that $|V_1|=|V_2|+1$. Let
$${\cal X}_0'=\{X_i  : i\in  [k], X_i  \mbox { crosses $H$, }  |X_i|\equiv 0
\mod   2,  \mbox{ and } ||X_i\cap V_1|-|X_i\cap V_2||=2\}$$
and
$${\cal X}_1=\{X_i: i\in  [k], \mbox{ $X_i$ crosses
  $H$, and } |X_i|\equiv 1 \mod 2 \}.$$
Let $W_1=\bigcup_{X_i\in {\cal
    X}_1}X_i$,  $W_0=\bigcup_{X_i\in {\cal X}_0'}X_i$, $|{\cal
  X}_1|=t$, and $|{\cal X}_0'|=t'$.
Then
\begin{itemize}
\item [$(i)$] $|W_1\cap V_1|-|W_1\cap V_2|=t$,
\item [$(ii)$]  $|W_0\cap V_1|-|W_0\cap V_2|=2t'$, and
\item [$(iii)$] $|\overline{W_1\cup W_0}\cap V_1|-|\overline{W_1\cup W_0}\cap V_2|=-(t+2t'-1)$.
\end{itemize}
\end{lem}
\begin{proof}
First, we prove $(i)$. If $W_1=\emptyset$, then $|W_1\cap V_1|-|W_1\cap V_2|=0=t$.
So assume $W_1\ne \emptyset$ and let $X_i\in {\cal X}_1$. Then $|X_i|$
is odd and $X_i$ crosses $H$. Hence, since $|V(G)|$ is odd, $|\overline{X_i}|$ is even.
By Lemma \ref{l1}, $|\overline{X_i}\cap V_1|=|\overline{X_i}\cap V_2|$.
Because $|V_1|=|V_2|+1$, $|X_i\cap V_1|-|X_i\cap V_2|=1$.
Hence,  $|W_1\cap V_1|-|W_1\cap V_2|=\sum_{X_i\in {\cal X}_1}(|X_i\cap V_1|-|X_i\cap V_2|)=t$.

We now prove $(ii)$. If $W_0=\emptyset$ then $t'=0$ and the result
holds trivially. So assume $W_0\ne \emptyset$ and let  $X_i\in {\cal X}_0'$.
Then $|X_i|$ is even and $X_i$ crosses $H$. Since $|V(G)|$ is odd, $|\overline{X_i}|$ is odd.
By Lemma \ref{l1}, $||\overline{X_i}\cap V_1|-|\overline{X_i}\cap V_2||=1$,
and by the definition of ${\cal X}_0'$, $||X_i\cap V_1|-|X_i\cap V_2||=2$.
Therefore, because $|V_1|=|V_2|+1$, $|X_i\cap V_1|-|X_i\cap V_2|=2$
(as well as $|\overline{X_i}\cap V_1|-|\overline{X_i}\cap V_2|=-1$).
Hence $|W_0\cap V_1|-|W_0\cap V_2|=\sum_{X_i\in {\cal X}'_0}(|X_i\cap V_1|-|X_i\cap V_2|)=2t'$.

It is easy to see that $(iii)$ follows from $(i)$, $(ii)$ and the assumption $|V_1|=|V_2|+1$.
\end{proof}

We now give a necessary and sufficient condition for a complete
multipartite graph with odd order to admit a good bisection.
Let $G=K_{r_1,\dots, r_k}$ with partition sets $X_i$, $i\in [k]$, such that
$|X_i|=r_i$. Let ${\cal X}=\{X_i: i\in [k]\}$, ${\cal S}_1=\{X_i:
i\in  [k] \mbox{ and }
|X_i| \equiv  1 \mod  2\}$ and ${\cal S}_0=\{X_i: i\in
[k] \mbox{ and } |X_i|\equiv 0 \mod  2\}$.
For any ${\cal A}\subseteq {\cal X}$, let $s({\cal A})=\sum_{X_i\in {\cal A}}|X_i|$. We say that ${\cal A}$ is {\it good} if
there exists  ${\cal A}'\subseteq {\cal A}$ such that $$s({\cal A}')=s({\cal A})/2+(m+2n-1)/2,$$
 where $m=|{\cal S}_1\setminus {\cal A}|$ and $n$ is a nonnegative integer with
 $n\leq |{\cal S}_0\setminus {\cal A}|$.

\begin{prop}\label{t2}
Let $G=K_{r_1,\dots,r_k}$ with partition sets $X_1, \ldots, X_k$, and assume $|V(G)|$ is odd.
Let ${\cal X}=\{X_i: i\in [k]\}$. Then $G$ has a good bisection if and only if ${\cal X}$ has a good subset.
\end{prop}

\begin{proof}
First, we prove that if $G$ has a good bisection, then ${\cal X}$ has a good subset.
Let $H$ be a good bisection of $G$ and let $V_1,V_2$ be the corresponding partition sets of $H$.
Since $|V(G)|$ is odd, we may assume that $|V_1|=|V_2|+1$. Let ${\cal
  S}=\{X_i :  i\in  [k] \mbox{ and } X_i \mbox{ crosses } H\}$, ${\cal
  X}_1=\{X_i: i\in [k], X_i  \mbox { crosses $H$,  and }  |X_i|\equiv 1 \mod
2\}$,  ${\cal X}'_0=\{X_i: i\in [k], X_i  \mbox { crosses $H$,  }
|X_i|\equiv 0 \mod
2, \mbox{ and } ||X_i\cap V_1|-|X_i\cap V_2||=2\}$, and ${\cal
  X}_0''=\{X_i: i\in [k], X_i  \mbox { crosses $H$, }  |X_i|\equiv 0 \mod
2, \mbox{ and } |X_i\cap V_1|=|X_i\cap V_2|\}$.

By Lemma \ref{l1} $(ii)$ and the assumption $|V_1|=|V_2|+1$, for every
$X_i\in {\cal X}$ crossing $H$ and $|X_i|$ even,
$||X_i\cap V_1|-|X_i\cap V_2||=2$ or $|X_i\cap V_1|=|X_i\cap
V_2|$. Hence, ${\cal S}={\cal X}_1\cup {\cal X}_0'\cup {\cal X}_0''$.
Let ${\cal A}={\cal X}\setminus {\cal S}$.

Moreover, let $W_1=\bigcup_{X_i\in {\cal X}_1}X_i$, $W_0=\bigcup_{X_i\in {\cal X}'_0}X_i$,
$W_0'=\bigcup_{X_i\in {\cal X}''_0}X_i$ and $W_2=\bigcup_{X_i\in  {\cal A}}X_i$.
Then $V(H)=W_1\cup W_0\cup W_0'\cup W_2$. By the definition of ${\cal X}''_0$, $|W_0'\cap V_1|=|W_0'\cap V_2|$.
Let $|{\cal X}_1|=t$ and  $|{\cal X}'_0|=t'$; then
by Lemma \ref{l2} $(iii)$, $|(W_0'\cup W_2)\cap V_2|-|(W_0'\cup W_2)\cap V_1|=t+2t'-1$.
Combining these two equalities, we get
\begin{equation}\label{e1}
 |W_2\cap V_2|-|W_2\cap V_1|=|(W_0'\cup W_2)\cap V_2|-|(W_0'\cup W_2)\cap V_1|=t+2t'-1.
\end{equation}

Since $X_i$ does not cross $H$ for any $X_i\subseteq {\cal A}$, there exists ${\cal A}'\subseteq {\cal A}$ such
that $W_2\cap V_2=\bigcup_{X_i\in {\cal A'}}X_i$ and $W_2\cap V_1=\bigcup_{X_i\in {\cal A}\setminus {\cal A}'}X_i$.
Now,  $|W_2\cap V_2|=s({\cal A}')$ and $|W_2\cap V_1|=s({\cal A})-s({\cal A}')$.
Thus $s({\cal A}')=s({\cal A})/2+(t+2t'-1)/2$ by (\ref{e1}).
Note that  $t=|{\cal X}_1|=|{\cal S}_1\setminus {\cal A}|$ and
$t'=|{\cal X}'_0|\leq|{\cal X}_0'\cup {\cal X}_0''|
=|{\cal S}_0\setminus {\cal A}|$.
So ${\cal A}$ is a good subset of ${\cal X}$.

\medskip

Now, we prove that if ${\cal X}$ has a good subset, then $G$ has a good bisection.
Let ${\cal A}$ be a good subset of ${\cal X}$.
Then there exists ${\cal A}'\subseteq {\cal A}$ such that $s({\cal A}')=s({\cal A})/2+(m+2n-1)/2$,
where $m=|{\cal S}_1\setminus {\cal A}|$ and $n\leq |{\cal
  S}_0\setminus {\cal A}|$.
Let ${\cal S}_0'\subseteq {\cal S}_0\setminus {\cal A}$ with $|{\cal
  S}_0'|=n$, and let ${\cal S}_0''=({\cal S}_0\setminus {\cal
  A})\setminus {\cal S}_0'$.
We partition $V(G)$ into $V_1$ and $V_2$ such that
\begin{itemize}
\item $|X_i\cap V_1|-|X_i\cap V_2|=1$ if $X_i\in {\cal S}_1\setminus {\cal A}$,
\item $|X_i\cap V_1|-|X_i\cap V_2|=2$ if $X_i\in {\cal S}_0'$,
\item $|X_i\cap V_1|-|X_i\cap V_2|=0$ if $X_i\subseteq {\cal S}_0''$,
  and
\item $X_i\subseteq V_1$ if $X_i\in {\cal A}\setminus {\cal A}'$, and
  $X_i\subseteq V_2$ if $X_i\in {\cal A}'$.
\end{itemize}
Then $$|V_1|-|V_2|=2|{\cal S}_0'|+|{\cal S}_1\setminus A|+s({\cal A})-2s({\cal
  A}')=1.$$

Let $H$ be the bisection of $G$ with partition sets $V_1$ and $V_2$
and edge set  $E(H)=\{uv\in E(G):  u\in V_1\mbox{ and } v\in V_2\}$.
Next, we show that $H$ is a good bisection of $G$.
Note that, for each $v\in X_i\subseteq V(G)$,
$d_G(v)=|\overline{X_i}|$,
$d_H(v)=|\overline{X_i}\cap V_1|$ if $v\in V_2$,  and $d_H(v)=
|\overline{X_i}\cap V_2|$ if $v\in V_1$.
Also note that
$$|\overline{X_i}\cap V_1|-|\overline{X_i}\cap V_2|=(|V_1|-|V_2|)-(|X_i\cap V_1|-|X_i\cap V_2|)=1-(|X_i\cap V_1|-|X_i\cap V_2|).$$

If $v\in X_i$ for some $X_i\in {\cal S}_1\setminus {\cal A}$, then
$|\overline{X_i}\cap V_1|-|\overline{X_i}\cap V_2|=1-1=0$; so
$d_H(v)=d_G(v)/2$.

If $v\in X_i$ for some $X_i\in {\cal S}_0'$ then
$|\overline{X_i}\cap V_1|-|\overline{X_i}\cap V_2|=1-2=-1$; so
$d_H(v)\ge (d_G(v)-1)/2$.

If $v\in X_i$ for some $X_i\in {\cal S}_0''$ then
$|\overline{X_i}\cap V_1|-|\overline{X_i}\cap V_2|=1-0=1$; so
$d_H(v)\ge (d_G(v)-1)/2$.

If $v\in X_i\cap V_2$ for some $X_i\in {\cal A}$ then
$|\overline{X_i}\cap V_1|-|\overline{X_i}\cap V_2|=1+|X_i|>0$; so
$d_H(v)=|\overline{X_i}\cap V_1|\ge d_G(v)/2$.

Finally, suppose $v\in X_i\cap V_1$ for some $X_i\in {\cal A}$. Then
$|\overline{X_i}\cap V_1|-|\overline{X_i}\cap V_2|=1-|X_i|$, i.e.,
$|\overline{X_i}\cap V_2|-|\overline{X_i}\cap V_1|=|X_i|-1\ge 0$. This
implies that $d_H(v)=|\overline{X_i}\cap V_2|\ge d_G(v)/2$.
\end{proof}



\begin{proof} [Proof of Proposition \ref{prop:BS-3partite}]
Let $G=K_{r_1,r_2,r_3}$ and ${\cal X}=\{X_1,X_2,X_3\}$ such that
$X_1,X_2,X_3$ are the partition sets of $G$ and $|X_i|=r_i$ for $i\in
[3]$.   Let ${\cal S}_0=\{X_i: i\in [3] \mbox{ and } |X_i|\equiv 0 \mod 2\}$ and
${\cal S}_1=\{X_i: i\in [3] \mbox{ and } |X_i|\equiv 1 \mod 2\}$. Then
${\cal S}_0=\emptyset$ and ${\cal S}_1={\cal X}$.

If ${\cal X}$ has no good subset then the assertion follows from
Proposition~\ref{t2}. So assume that ${\cal A}$ is a good subset of
${\cal X}$ with ${\cal A}'\subseteq {\cal A}$ such that $s({\cal
  A}')=s({\cal A})/2+(m+2n-1)/2$, where $m=|{\cal S}_1\setminus {\cal
  A}|=3-|{\cal A}|$ and
$n\le |{\cal S}_0\setminus {\cal A}|=0$.
So  $$s({\cal A}')=s({\cal A})/2-|{\cal A}|/2 +1.$$

It is easy to see that ${\cal A}\ne \emptyset$.  Since $r_i\ge 3\ge
|{\cal A}|$,  ${\cal A}'\ne \emptyset$ and, hence, $|{\cal A}|\ne
1$. Since $r_1,r_2,r_3$ are all distinct, $|{\cal A}|\ne 2$. So
$|{\cal A}|=3$. Now a straightforward analysis shows that
for some $i\in [3]$,   $r_i\in\{\lfloor(r_1+r_2+r_3)/2\rfloor,
\lceil(r_1+r_2+r_3)/2\rceil\}$. This is a contradiction.
\end{proof}



Proposition  \ref{t2} characterizes  those complete  multipartite graphs
which do not have a good bisection. The next result says that there are more such
examples.

\begin{prop}\label{t3}
Let $G=K_{r_1,\dots,r_k}$ where $r_i\geq 7$ for every $i\in [k]$.
Suppose $G$ does not have a good bisection. Then for any edge $e\in E(G)$,
 $G-e$ does not have a good bisection.
\end{prop}

\begin{proof}
Assume, to the contrary, that $G'=G-e$ has a good bisection $H'$ with
partition sets $V_1, V_2$.
We may assume that $E(H')=\{xy\in E(G') :  x\in V_1 \mbox{ and } y \in V_2\}$.
Then for every vertex $v\in V(G')$,
$d_{H'}(v)\geq (d_{G'}(v)-1)/2$.
Let $H$ be the bisection of $G$ with partition sets $V_1$ and $V_2$
such that  $E(H)=\{xy\in E(G):  x\in V_1 \mbox{ and } y\in V_2\}$.
Let $e=uw$.

Then $d_{G}(v)=d_{G'}(v)$ for all $v\in V(G)\setminus\{u,w\}$,
and $d_{G}(v)=d_{G'}(v)+1$ for $v\in \{u,w\}$.
Also, we have $d_{H}(v)=d_{H'}(v)$ for all $v\in V(G)\setminus\{u,w\}$,
and $d_{H'}(v)\leq d_{H}(v)\leq d_{H'}(v)+1$ for $v\in \{u,w\}$.

Since $H$ is not a good bisection of $G$, there exists a vertex $v\in V(G)$ such that
$d_{H}(v)<(d_{G}(v)-1)/2$.
So we have
$$
(d_{G'}(v)-1)/2\leq d_{H'}(v)\leq
d_{H}(v)<(d_{G}(v)-1)/2\leq d_{G'}(v)/2$$
which implies that $d_{G'}(v)$ is odd (since $d_H(v)$ is an integer),
$d_{G}(v)=d_{G'}(v)+1$, and
$(d_{G'}(v)-1)/2=d_{H'}(v)=
d_{H}(v)$. Since
$d_{G}(v)=d_{G'}(v)+1$,  $v\in\{u,w\}$.

Assume, without loss of
generality, that $v=u\in X_i\cap V_1$, where $X_1,\ldots, X_k$ are the
partition sets of $G$. (Then $w\notin X_i$ as $uw\in E(G)$.)
So $d_{G'}(u)=|\overline{X_i}|-1$, $|\overline{X_i}|$ is even, and
$w\in V_1$.
Thus,
$|V_2\cap\overline{X_i}|=d_{H'}(u)=(d_{G'}(u)-1)/2=|\overline{X_i}|/2-1.$
Therefore,
$
|V_1\cap\overline{X_i}|=|\overline{X_i}|-|V_2\cap\overline{X_i}|=|\overline{X_i}|/2+1.
$
So $|V_1\cap\overline{X_i}|-|V_2\cap\overline{X_i}|=2$.
Because $||V_1|-|V_2||=1$, $||V_1\cap X_i|-|V_2\cap X_i||\leq 3$.
Since $|X_i|\geq7$, $|V_1\cap X_i|\geq 2$.
Therefore, there exists a vertex $v_1\in V_1\cap X_i$ such that $v_1\not=u$.
Also, $v_1\ne w$ since $w\notin X_i$.
Then $d_{G'}(v_1)=d_{G}(v_1)=|\overline{X_i}|$ is even.
Thus $d_{H'}(v_1)=|V_2\cap\overline{X_i}|=|\overline{X_i}|/2-1
<(d_{G'}(v_1)-1)/2$.
This contradicts the assumption that $H'$ is a good bisection of $G'$.
\end{proof}

\section{Scott's questions on bipartitions}
In this section, we address two questions of Scott \cite{Sc05} on bipartitions of graphs.
First, we prove Theorem \ref{thm:norm} on $\ell_\lambda$-norm of
bipartitions (with $\lambda\ge 1$),  for which
we need a result of Bollob\'{a}s and Scott \cite{bs99} on judicious
bipartitions. Recall the definition of $t(m)=\sqrt{m/2+1/16}-1/4$.

\begin{lem} [Bollob\'{a}s and Scott] \label{thm:BS99}
Let $G$ be a graph with $m$ edges. Then there exists a
bipartition $V(G)=V_1\cup V_2$ such that $e(V_1,V_2)\ge m/2+t(m)/2$ and $\max\{e(V_1),e(V_2)\}\le m/4+t(m)/4$.
Moreover, if for every such bipartition $V(G)=V_1\cup V_2$ it always holds that $\max\{e(V_1),e(V_2)\}=m/4+t(m)/4$,
then $G$ must be a complete graph of odd order.
\end{lem}


\begin{proof}[Proof of Theorem \ref{thm:norm}.]
Let $\lambda\ge 1$ and let $G$ be a graph with $m$ edges.
By Lemma~\ref{thm:BS99}, there is a bipartition $V_1,V_2$ of $V(G)$ such
that $e(V_1,V_2)\ge m/2+t(m)/2$ and, for $i\in [2]$, $e(V_i)\le
m/4+t(m)/4=\binom{t(m)+1}{2}$.
Note that $$e(V_1)+e(V_2)=m-e(V_1,V_2)\le m/2-t(m)/2=t(m)^2.$$
Without loss of generality, we assume that $e(V_1)\ge e(V_2)$. Then
$e(V_2)\le t(m)^2/2$.

We claim that $e(V_1)^\lambda+e(V_2)^\lambda\le {t(m)\choose 2}^\lambda+{t(m)+1\choose 2}^\lambda.$
This is true if  $e(V_2)\le \binom{t(m)}{2}$. So we may assume
$\binom{t(m)}{2}\le e(V_2)\le t(m)^2/2$. For $\lambda \geq 1$,
the function $f(x)=(t(m)^2-x)^\lambda+x^\lambda$ is strictly decreasing when $\binom{t(m)}{2}\le x\le t(m)^2/2.$
Therefore,
$$e(V_1)^\lambda+e(V_2)^\lambda \le (t(m)^2-e(V_2))^\lambda+e(V_2)^\lambda \le \binom{t(m)}{2}^\lambda +\binom{t(m)+1}{2}^\lambda.$$


Now assume that for every bipartition $V_1,V_2$ of $V(G)$, we have
$e(V_1)^\lambda+e(V_2)^\lambda= {t(m)\choose 2}^\lambda+{t(m)+1\choose
  2}^\lambda.$ Then it follows from the above arguments, $e(V_2)=\binom{t(m)}{2}$ and  $e(V_1)+e(V_2)=t(m)^2$.
So $\max\{e(V_1),e(V_2)\}=e(V_1)=t(m)^2-\binom{t(m)}{2}=m/4+t(m)/4$.
By Lemma~\ref{thm:BS99}, $G$ is a complete graph of odd order.
\end{proof}

\noindent {\bf Remark.} From the above proof, we see that actually
$V(G)$ has a bipartition $V_1, V_2$ such that $e(V_1)^\lambda+e(V_2)^\lambda\le  {t(m)\choose 2}^\lambda+{t(m)+1\choose
  2}^\lambda$ for all  $\lambda\ge 1$.

\medskip

To extend Theorem \ref{thm:norm} to $k$-partitions for  $k\ge
3$, we need the following result of Xu and Yu \cite{XY} on
$k$-partitions.
\begin{lem} [Xu and Yu] \label{xuyu-k}
Let $G$ be a graph with $m$ edges and let $k\ge 3$ be a positive
integer. Then there exists a $k$-partition $V(G)=V_1\cup ...\cup V_k$ such that
$$e(V_1,\ldots, V_k)\ge
\frac{k-1}{k}m+\frac{k-1}{k}t(m)-\frac{17k}{8},$$
and for $i\in [k]$,
$$e(V_i)\le \frac{m}{k^2}+\frac{k-1}{k^2}t(m).$$
\end{lem}

We now determine the $\ell_\lambda$-norm (where $\lambda\ge 1$) for $k$-partitions,
up to an additive term $O(m^{\lambda-1})$. The proof is similar to the bipartition case.

\begin{theo}\label{thm:norm-kpar}
Let $k\ge 3$ be an integer and $\lambda\ge 1$ be a real number. Then any graph $G$ with $m$ edges has a $k$-partition
$V(G)=V_1\cup ...\cup V_k$ such that
$$\sum_{i=1}^k e(V_i)^\lambda\le \frac{1}{k^{2\lambda-1}}m^\lambda -
\frac{(k-1)\lambda}{k^{2\lambda-1}} m^{\lambda-1} t(m)+O(m^{\lambda-1}).$$
\end{theo}

\begin{proof}
Let $G$ be a graph with $m$ edges.
By Lemma~\ref{xuyu-k}, there exists a $k$-partition $V(G)=V_1\cup ...\cup V_k$ such that
$$\sum_{i=1}^k e(V_i)=m-e(V_1,\ldots, V_k)\le  \frac{m}{k}-\frac{k-1}{k}t(m)+\frac{17k}{8}$$
and for $i\in [k]$,
$$e(V_i)\le \frac{m}{k^2}+\frac{k-1}{k^2}t(m).$$
Without loss of generality, let $e(V_1)\ge e(V_2)\ge \ldots \ge
e(V_k)$ and let $\alpha:=e(V_1)-m/k^2$.

If $\alpha\le -\frac{k-1}{k^2}t(m)$ then
$$\sum_{i=1}^k e(V_i)^\lambda\le k\left(\frac{m}{k^2}-\frac{k-1}{k^2}t(m)\right)^\lambda\le \frac{1}{k^{2\lambda-1}} m^\lambda-
\frac{(k-1)\lambda}{k^{2\lambda-1}} m^{\lambda-1}t(m)+O(m^{\lambda-1}).$$
So we may assume that
\begin{equation}\label{e3.6}
-\frac{k-1}{k^2}t(m)\le \alpha\le \frac{k-1}{k^2}t(m).
\end{equation}

Note that we may  assume $\sum_{i=1}^k
e(V_i)=\frac{m}{k}-\frac{k-1}{k}t(m)+\frac{17k}{8}$. Also note
that  $\sum_{i=1}^k e(V_i)^\lambda$ increases if we replace $e(V_k)$
by $e(V_k)-1$ and replaces $e(V_i)$ by $e(V_i)+1$,  for any $i\in [k-1]$.
Therefore, we may further assume that
$e(V_1)=...=e(V_{k-1})=\frac{m}{k^2}+\alpha$ and
$e(V_k)=\frac{m}{k^2}-\frac{k-1}{k}t(m)+\frac{17k}{8} -(k-1)\alpha$.
So by \eqref{e3.6}, we have
\begin{equation*}
\begin{split}
\sum_{i=1}^k e(V_i)^\lambda&\le (k-1) \left(\frac{m}{k^2}+\alpha\right)^\lambda+\left(\frac{m}{k^2}-\frac{k-1}{k}t(m)+\frac{17k}{8}-(k-1)\alpha\right)^\lambda\\
&=(k-1)\left(\frac{m}{k^2}+\alpha\right)^\lambda+\left(\frac{m}{k^2}-\frac{k-1}{k}t(m)-(k-1)\alpha\right)^\lambda+O(m^{\lambda-1})\\
&\le (k-1)\left(\frac{m}{k^2}+\frac{k-1}{k^2}t(m)\right)^\lambda +\left(\frac{m}{k^2}-\frac{k-1}{k}t(m)-\frac{(k-1)^2}{k^2}t(m)\right)^\lambda+O(m^{\lambda-1})\\
&=\frac{1}{k^{2\lambda-1}} m^\lambda-\frac{(k-1)\lambda}{k^{2\lambda-1}}m^{\lambda-1} t(m)+O(m^{\lambda-1}),
\end{split}
\end{equation*}
where the second inequality holds because the expression in the second
line is an increasing function of $\alpha$, for $-\frac{k-1}{k^2}t(m)\le \alpha\le \frac{k-1}{k^2}t(m)$.
\end{proof}

We remark that the bound in the above theorem  is tight up to the term
$O(m^{\lambda-1})$,
by considering the complete graph $K_{ks}$ which has $m=\binom{ks}{2}$ edges.
Thus $s=(2t(m)+1)/k$.
The minimum $\sum_{i=1}^k e(V_i)^\lambda$ over all $k$-partitions
$V_1, \ldots,V_k$ of $V(K_{ks})$
is attained when $|V_i|=s$ for $i\in [k]$, and this minimum value equals
$$k\binom{s}{2}^\lambda=\frac{k}{2^\lambda}\left(\frac{2t(m)+1}{k}\right)^\lambda  \left(\frac{2t(m)+1-k}{k}\right)^\lambda.$$
Using $2t(m)^2+t(m)=m$ and $t(m)=\Theta(\sqrt{m})$, we see that
$$k \binom{s}{2}^\lambda=\frac{1}{k^{2\lambda-1}} m^\lambda- \frac{(k-1)\lambda}{k^{2\lambda-1}}m^{\lambda-1}t(m)+\Theta(m^{\lambda-1}).$$

It would be interesting to find the optimal upper bound in Theorem \ref{thm:norm-kpar}.
We believe that the extremal graphs for $\ell_\lambda$-norms of $k$-partitions (where $\lambda\ge 1$)
should be the complete graphs $K_{kn+\lfloor k/2\rfloor}$.
We formulate the following question.
For fixed $\lambda\ge 1$ and integer $k\ge 2$, let $s:=s(m)$ be such that $m=\binom{ks+\lfloor k/2\rfloor}{2}$ and let
$$f_{\lambda,k}(m):= \left\lfloor \frac{k}{2}\right\rfloor\binom{s+1}{2}^\lambda+ \left\lceil \frac{k}{2}\right\rceil\binom{s}{2}^\lambda.$$

\begin{ques}
Fix any real $\lambda\ge 1$ and integer $k\ge 2$.
For any positive integer $m$, is it true that
$$\min_{V(G)=V_1\cup...\cup V_k} \sum_{i=1}^k e(V_i)^\lambda\le f_{\lambda,k}(m) $$
for all graphs $G$ with $m$ edges,
with equality if and only if $m=\binom{ks+\lfloor k/2\rfloor}{2}$ for some integer $s$?
Does the equality hold only for $K_{ks+\lfloor k/2\rfloor}$ (modulo some isolated vertices)?
\end{ques}
\noindent A result of Bollob\'as and Scott \cite{bs02} shows that this is true for $\lambda=1$ and any $k$.
Theorem \ref{thm:norm} provides an affirmative answer for the case $k=2$.

\medskip

We now turn to  the following question of Scott \cite{Sc05}.

\begin{ques}
Does every graph $G$ with
$\binom{kn}{2}$ edges have a vertex partition into $k$ sets, each of
which contains at most $\binom{n}{2}$ edges?
\end{ques}
We give a negative answer to this question in the case $k=2$. For this
we need to show that there exist an infinite sequence of pairs of
integers with certain properties.

\begin{lem}\label{lem:n+t}
There are pairs $(a_i,b_i)$ of integers for all $i\ge 0$ such that
\begin{itemize}
\item [$(i)$] $a_i\ge 36$ and $a_i$ is even, and $b_i\ge 21$ and $b_i$ is odd,
\item [$(ii)$] $3b_i(b_i-1)=a_i(a_i-1)$, and
\item [$(iii)$] $b_i\le 7a_i/12$.
\end{itemize}
\end{lem}
\begin{proof}
We recursively define integer pairs $(n_i, t_i)$ as follows, such that
the desired sequence $\{(a_i,b_i)\}_{i\ge 0}$ will be a subsequence of
  $\{(n_i,t_i)\}_{i\ge 0}$.

Let $$(n_0, t_0)=(36,21) \text{~~ and~~} (n_1, t_1)=(133,77)$$
and, for $i\ge 1$, let
\begin{equation}\label{equ:ni-ti}
n_{i+1}=4n_i-n_{i-1}-1 \text{~~ and~~} t_{i+1}=4t_i-t_{i-1}-1.
\end{equation}
For convenience, we write $$\alpha_i:=n_i(n_i-1)-3t_i(t_i-1)$$ for
$i\ge 0$, and
$$\beta_i:=2n_in_{i-1}-n_i-n_{i-1}-6t_it_{i-1}+3t_i+3t_{i-1}+1$$
for $i\ge 1$.

We claim that $\alpha_i=0$ for $i\ge 0$ and that $\beta_i=0$ for $i\ge 1$.
By a direct calculation, we see that $\alpha_0=0$, $\alpha_1=0$ and $\beta_1=0$.
Now assume for some $i\ge 1$, we have  $\alpha_j=0$ for $j\in [i]\cup
\{0\}$, and $\beta_j=0$ for $j\in [i]$.
Using \eqref{equ:ni-ti} and the definition of $\alpha_{i+1}$ and
$\beta_{i+1}$, we have
\begin{equation*}
\begin{array}{lll}
\alpha_{i+1}&=(4n_{i}-n_{i-1}-1)(4n_{i}-n_{i-1}-2)-3(4t_{i}-t_{i-1}-1)(4t_{i}-t_{i-1}-2)\\
&=16n_{i}^2-8n_{i}n_{i-1}+n_{i-1}^2-12n_{i}+3n_{i-1}+2\\
&~~ -48t_{i}^2+24t_{i}t_{i-1}-3t_{i-1}^2+36t_{i}-9t_{i-1}-6\\
&= 16 \alpha_{i-1}+\alpha_{i-2}-4\beta_{i-1}=0,\\
~\\
\beta_{i+1}&=2 (4n_{i}-n_{i-1}-1)n_{i}-(4n_{i}-n_{i-1}-1)-n_{i} \\
&~~ -6(4t_{i}-t_{i-1}-1)t_{i}+3(4t_{i}-t_{i-1}-1)+3t_{i}+1  \\
&=8n_{i}^2-2n_{i-1}n_{i}-7n_{i}+n_{i-1}-24t_{i}^2+6t_{i-1}t_{i}+21t_{i}-3t_{i-1}-1\\
&=8\alpha_{i-1}-\beta_{i-1}=0.
\end{array}
\end{equation*}
Thus, the claim follows from induction.

From \eqref{equ:ni-ti}, we see that both $\{n_i\}_{i\ge 0}$ and
$\{t_i\}_{i\ge 0}$ are increasing sequences; so $n_i\ge 36$ and $t_i\ge 21$ for  $i\ge 0$.
Moreover,  $t_i\le 7n_i/12$ for $i\ge 0$. For otherwise,
$t_i>7n_i/12$ for some $i$. Then $i\ge 1$ and
$$3t_i(t_i-1) >3 (7n_i/12)(7n_i/12-1)
=49n_i^2/48-7n_i/4,$$
which is larger than $n_i(n_i-1)$ (since $n_i\ge 36$), a contradiction.

Using \eqref{equ:ni-ti}, it is easy to observe that
$n_i$ is even if and only if $i\equiv 0,3 \mod 4$,
and that $t_i$ is odd if and only if $i\equiv 0,1 \mod 4$.
Therefore letting $a_i=n_{4i}$ and $b_i=t_{4i}$ for $i\ge 0$, we see
that the pairs $(a_i, b_i)$ satisfy all requirements $(i)$, $(ii)$ and
$(iii)$.
\end{proof}

We now give the proof of Theorem \ref{thm:2n}.

\begin{proof}[Proof of Theorem \ref{thm:2n}.]
By Lemma \ref{lem:n+t}, there exist infinitely many  pairs $(2n,t)$
of positive integers such that $t$ is odd,  $t\le 7n/6$, and
$3t(t-1)=2n(2n-1)$.

Let $G$ be the union of three pairwise disjoint copies of the clique $K_t$. Then
$|V(G)|=3t$ and $$e(G)=\frac{3t(t-1)}{2}=n(2n-1)=\binom{2n}{2}.$$
Let $V_1,V_2$ be a bipartition of $V(G)$. Without loss of generality, we assume that $|V_1|\ge |V_2|$.

Then $G[V_1]$ is the disjoint union of three cliques, say  $K_a,K_b,$ and
$K_c$. (We set $K_0=\emptyset$.) Hence, $G[V_2]$ is the disjoint union of cliques $K_{t-a},K_{t-b}$ and $K_{t-c}$.
As $t$ is odd, we have $$a+b+c\ge \left\lceil
  |V(G)|/2\right\rceil=(3t+1)/2.$$
Choose integers $a',b',c'$ such that $a'\le a, b'\le b, c'\le c$ and $$a'+b'+c'=(3t+1)/2.$$
We also need an easy property of binomial coefficients that for any integers $m-n\ge 2$,
\begin{equation}\label{equ:binoms}
\binom{m}{2}+\binom{n}{2}> \binom{m-1}{2}+\binom{n+1}{2}.
\end{equation}
Then we have
\begin{equation*}
\begin{split}
e(V_1)&=\binom{a}{2}+\binom{b}{2}+\binom{c}{2}\\
& \geqslant \binom{a'}{2}+\binom{b'}{2}+\binom{c'}{2}\\
&\geqslant
\binom{(t+1)/2}{2}+\binom{(t+1)/2}{2}+\binom{(t-1)/2}{2}
\quad \mbox{(by \eqref{equ:binoms})}\\
&=\frac{3t^2-4t+1}{8}\\
&=\frac{4n^2-2n-t+1}{8} \quad \mbox{ (as $3t(t-1)=2n(2n-1)$)}\\
&\geqslant \binom{n}{2}+\frac{5}{48}n \quad \mbox{ (as $t\le 7n/6$).}
\end{split}
\end{equation*}
This completes the proof of Theorem \ref{thm:2n}.
\end{proof}

It seems likely that similar result holds for general $k$-partitions,
though we are not able to construct such graphs due to difficulties in
proving a more general version of Lemma~\ref{lem:n+t}.


\end{document}